\documentclass[11pt,a4paper]{article}
\usepackage{amsmath}
\usepackage{amsfonts}
\usepackage{latexsym,amsthm,amscd}

\newtheorem{thm}{Theorem}
\newtheorem{prop}[thm]{Proposition}
\newtheorem{defn}[thm]{Definition}
\newtheorem{lem}[thm]{Lemma}
\newcounter{alphthm}
\usepackage[colorlinks]{hyperref}
\usepackage{color}
\usepackage{graphicx}
\usepackage{graphics}
\usepackage{makeidx}
\usepackage{showidx}
\usepackage{latexsym}
\usepackage{amssymb}
\usepackage{verbatim}
\usepackage{amsmath}
\usepackage{amsthm}
\usepackage{amsfonts}
\usepackage{amssymb,amsmath}
\usepackage{latexsym,amsthm,amscd}
\usepackage[all]{xy}

\newcommand{\be}{\begin{equation}}
\newcommand{\ee}{\end{equation}}
\newcommand{\ben}{\begin{enumerate}}
\newcommand{\een}{\end{enumerate}}
\newcommand{\beq}{\begin{eqnarray}}
\newcommand{\eeq}{\end{eqnarray}}
\newcommand{\beqn}{\begin{eqnarray*}}
\newcommand{\eeqn}{\end{eqnarray*}}

\newcommand{\pa}{\partial}
\newcommand{\bpf}{\begin{proof}}
\newcommand{\epf}{\end{proof}}
\newcommand{\bl}{\begin{lem}}
\newcommand{\el}{\end{lem}}
\newcommand{\bp}{\begin{prop}}
\newcommand{\ep}{\end{prop}}
\newcommand{\bd}{\begin{Def}}
\newcommand{\ed}{\end{Def}}
\newcommand{\bt}{\begin{thm}}
\newcommand{\et}{\end{thm}}
\newcommand{\R}{I\!\! R}
\def\nn{\nonumber}

\title{Circle-preserving transformations in Finsler spaces}
\author{Behroz Bidabad\thanks{ Visiting Professor
at the Indiana University Purdue University Indiana, emails: bbidabad@iupui.edu; bidabad@aut.ac.ir.}, Zhongmin Shen\thanks{ Supported in part by a NSF grant (DMS-0810159),email: zshen@math.iupui.edu.}}
\date{\footnotesize Department of Mathematics, Amirkabir University of Technology (Tehran Polytechnic), Tehran 15914, Iran.\\
Department of Mathematical Sciences, Indiana University Purdue University Indianapolis (IUPUI), Indianapolis, IN 46202, USA.}
\begin{document}
\maketitle

\begin{abstract}
Here, by extending the  definition of  circle to Finsler geometry,  we show that, every circle-preserving local diffeomorphism  is conformal. This result implies that in Finsler geometry, the definition of concircular change of metrics, a priori, does not require the conformal assumption.
\end{abstract}
\section*{Introduction.}
 In Riemannian geometry Vogel proved that every circle-preserving diffeomorphism is conformal, c.f. \cite{Vo} and \cite{Ku}. This theorem has been extended to pseudo-Riemannian manifolds in  \cite{Cat}. Here, we shall extend this theorem to Finsler manifolds. Using the Cartan covariant derivative along a curve, the  definition of  circles in a Finsler manifold is given. This definition is a natural extension of Riemannian one, see for instance, \cite{NY}. Some typical examples of circles are
 helices on a cylinder or a torus.  It should be remarked that these  circles need not to be closed in general, although it may be closed in some cases as on a torus.

 A geodesic circle in a Riemannian geometry, as well as in Finsler geometry, is a curve for which the first Frenet curvature $k_1$ is constant and the second curvature $k_2$ vanishes. In the other words a geodesic circle is a torsion free constant curvature curve. A \emph{concircular} transformation is defined by \cite{Ya} and \cite{Fi} in Riemannian geometry to be a conformal transformation which preserves geodesic circles.

 This notion has been similarly  developed in Finsler geometry by Agrawal and Izumi, cf. \cite{Ag,Iz,Iz2} and studied in \cite{AB,B} by one of the present authors.

The results obtained in this paper shows  that in the definition of concircular transformations, a priori, the conformal assumption is not necessary. That is to say, if a transformation preserves geodesic circles then it is conformal.

\section{Preliminaries}
 Let $M$ be a real n-dimensional  manifold of class $C^ \infty$. We
denote by  $TM\rightarrow M$ the  bundle
  of tangent vectors and by $ \pi:TM_{0}\rightarrow M$ the fiber bundle of
non-zero tangent vectors.
  A {\it{Finsler structure}} on $M$ is a function
$F:TM \rightarrow [0,\infty )$, with the following properties: (I)
$F$ is differentiable ($C^ \infty$) on $TM_{0}$; (II) $F$ is
positively homogeneous of degree one in $y$, i.e.
 $F(x,\lambda y)=\lambda F(x,y),  \forall\lambda>0$, where we denote
 an element of $TM$ by $(x,y)$.
(III) The Hessian matrix of $F^{2}/2$ is positive definite on
$TM_{0}$; $(g_{ij}):=\left({1 \over 2} \left[
\frac{\partial^{2}}{\partial y^{i}\partial y^{j}} F^2
\right]\right).$
A \textit{Finsler
manifold} $(M,g)$ is a pair of a differentiable manifold $M$ and a  tensor field $g=(g_{ij})$ on $TM$ which defined by a Finsler structure $F$. The spray of a Finsler structure $F$ is a vector field on $TM$
\[ G=y^i \frac{\pa }{\pa x^i} -2 G^i \frac{\pa }{\pa y^i},\]
where
\[ G^i =\frac{g^{il}}{4} \Big \{ \frac{\pa^2 F^2}{\pa x^m \pa y^l} y^m -\frac{\pa F^2}{\pa x^l} \Big \},\]
and $(g^{ij}):=(g_{ij})^{-1}$.

Let $(M,g)$ be a ${\cal C}^\infty$ Finsler manifold  and let $c$ be an oriented ${\cal C}^\infty$ parametric curve on
$M$ with equation $x^i(t)$.   We choose the pair $(x,\dot x)$, to be the line element
 along the curve $c$.

 Let $(x^i,y^i)$ be the local coordinates  on the slit tangent bundle $TM/0$. Using a Finsler connection  we can choose the natural basis $(\delta /\delta x^i,\partial/\partial y^i)$, where  $\frac{\delta }{\delta x^{j}}:=\frac{\partial }{\partial x^{j}}-N^i_j\frac{\partial }{\partial y^{i}}$, and $N^i_j :=\frac{1}{2}\frac{\pa G^i}{\pa y^j}$. The dual basis is given by $(d x^i,\delta y^i)$, where  $\delta y^k:=dy^k+N^k_ldx^l$.

Let $X$ be a ${\cal C}^\infty$ vector field
$X=X^i(t)\frac{\partial}{\partial x^i}|_{c(t)}$ along $c(t)$. We denote the Cartan covariant derivative of $X$ in direction of $\dot c = \frac{dx^j}{dt}\frac{\partial}{\partial x^j}$ by $\nabla_{_{\dot c}}X = \frac{\delta X^i}{dt} \frac{\pa }{\pa x^i}|_{c(t)}$. The components
$\frac{\delta X^i}{dt}$ can be determined explicitly as follows.
\be \label{Eq;CovDerAlongCurves2}
\nabla_{_{\dot c}}X =\nabla_{_{\dot c}}X^i\frac{\partial}{\partial x^i}=\frac{d X^i}{ dt}\frac{\partial}{\partial x^i}+X^i \nabla_{\dot c}\frac{\partial}{\partial x^i}.
\ee
The last term  in Eq. (\ref{Eq;CovDerAlongCurves2}) is given by $\nabla_{{\dot c}}\frac{\partial}{\partial x^i}:=\omega_i^j(\dot c)\frac{\partial}{\partial x^j}$, where $\omega_i^j(\dot c):=(\Gamma_{ik}^jdx^k+C^j_{ik}\delta y^k)(\dot c)$, is the connection $1$-form of Cartan connection, cf. \cite{BCS}, page 39. Here,  the coefficients $\Gamma^{i}_{jk}$ are Christoffel symbols with respect to the horizontal partial derivative  $\frac{\delta }{\delta x^{j}}$, that is,
     $$\Gamma^{i}_{jk}:=\frac{1}{2}g^{ih}(\frac{\delta g_{hk}}{\delta x^{j}}+ \frac{\delta g_{hj}}{\delta x^{k}}- \frac{\delta g_{jk}}{\delta x^{h}}),$$
      and  $C_{hk}^{i}:=\frac{1}{2}g^{im}\frac{\pa g_{mk}}{\pa y^h}$, is the \emph{Cartan torsion tensor}.
  Plugging $\delta y^k$ in $\omega_i^j(\dot c)$
\begin{eqnarray*}
\omega_i^j(\dot c)&=&(\Gamma_{ik}^jdx^k+C^j_{ik}(d y^k+N^k_ldx^l)(\dot c),\\
&=&(\Gamma_{ik}^jdx^k+C^j_{is}N^s_kdx^k)(\frac{dx^l}{dt}\frac{\partial}{\partial x^l}),\\
&=&(\Gamma_{il}^j+C^j_{is}N^s_l)(\frac{dx^l}{dt}),
\end{eqnarray*}
and replacing the resulting term  in Eq. (\ref{Eq;CovDerAlongCurves2}), we obtain
 the components of Cartan covariant derivative of $X$ in direction of $\dot c$.
\be\label{CartanCovDerAlongC1}
\frac{\delta}{ dt} X^i  = \frac{d X^i}{ dt}+ (\Gamma_{kh}^i+C^i_{ks}N^s_h)X^k\frac{dx^h}{dt}.
\ee
The Cartan covariant derivative $\nabla_{\dot c}$ is  metric-compatible  along $c$, that is, for any vector fields $X$ and $Y$ along $c$,
 \be \nn
 \frac{d}{dt} g(X,Y)= g(\nabla_{\dot c}X, Y)+g(X, \nabla_{\dot c}Y).
 \ee
More details about this preliminaries may be found in \cite{AZ1,BCS,Sh}.

\section{Circles in a Finsler manifold}

As a natural extension of circles in Riemannian geometry, cf. \cite{NY}, we  recall the  definition of a generalized circle in a Finsler manifold, called here simply, circle.
\begin{defn}
 Let $(M,g)$ be a Finsler manifold  of class $C^\infty$.  A smooth curve $c:I\subset \R\rightarrow M$ parameterized by arc length $s$ is called a \emph{circle} if there exist a unitary vector field $Y=Y(s)$ along $c$ and a positive constant $k$ such that
 \begin{eqnarray}
\label{Eq;cov der of X 1}\nabla_{c'} X= k Y,\\
\label{Eq;cov der of X 2}\nabla_{c'} Y= - k X,
 \end{eqnarray}
 where, $X:=c'=dc/ds$ is the unitary tangent vector field at each point $c(s)$.
  The number $1/k$ is called the \emph{radius} of the circle.
\end{defn}
Comparing this definition of circle with definition of a geodesic circle in Finsler geometry, recalled in the introduction, we find out that if in the definition of a geodesic circle we exclude the trivial case, $k_1=0$, that is, if we remove  geodesics, then we obtain the  definition of a  circle in a Finsler manifold.

\begin{lem}\label{prop;diff eq geod circ}
Let $c=c(s)$  be a unit speed  curve on an $n$-dimensional Finsler manifold $(M,g)$. If $c$ is a circle, then it  satisfies the following ODE
 \be\label{Eq;circle equ+1}
 \nabla_{c'}\nabla_{c'} X+ g(\nabla_{c'}X,\nabla_{c'}X) X=0,
 \ee
  where, $ g( \  ,\  )$ denotes scalar product determined by the tangent vector $c'$.
Conversely, $c $ satisfies (\ref{Eq;circle equ+1}), then it is either a geodesic or a circle.
\end{lem}
\begin{proof} Assume that  $c$ is a circle parameterized by arc-length. By means of metric compatibility  we have
$$ g(\nabla_{c'} X, X) = \frac{1}{2}\frac{d}{ds} [ g(X, X) ] =0.$$
 Eqs. (\ref{Eq;cov der of X 1}) and (\ref{Eq;cov der of X 2}) yield
\be
\nabla_{c'}\nabla_{c'}X= k \nabla_{c'} Y= -k^2 X.\label{Eq;circle equ+2}
\ee
This implies
\[ k^2 = - g(\nabla_{c'}\nabla_{c'}X, X) = \frac{d}{ds} [ g(\nabla_{c'} X, X) ] + g(\nabla_{c'} X, \nabla_{c'} X).\]
Plugging it into (\ref{Eq;circle equ+2}), we obtain (\ref{Eq;circle equ+1}).

Conversely, assume that $c=c(s)$ is a unit speed curve on $M$ which satisfies Eq. (\ref{Eq;circle equ+1}).  Then by metric compatibility property of $\nabla_{c'}$, we have
\begin{eqnarray}\label{Eq;metric comp 2}
\frac{d}{ds} g( \nabla_{c'}X,\nabla_{c'}X ) = 2 g(\nabla_{c'}\nabla_{c'}X, \nabla_{c'}X).
\end{eqnarray}
Plugging Eq. (\ref{Eq;circle equ+1}) into this equation we have
\begin{eqnarray}\label{Eq;metric comp 3}
  g(\nabla_{c'}\nabla_{c'}X, \nabla_{c'}X) = - g(\nabla_{c'}X, \nabla_{c'}X) g(X, \nabla_{c'}X).
\end{eqnarray}
Taking into account Eqs. (\ref{Eq;metric comp 2}) and (\ref{Eq;metric comp 3}) and the fact that $g(X, \nabla_{c'} X)=0$ for unitary tangent vector field $X$, we have
\begin{eqnarray*}
\frac{d}{ds} g( \nabla_{c'}X,\nabla_{c'}X ) = 0.
\end{eqnarray*}
Therefore $k^2:=g( \nabla_{c'}X,\nabla_{c'}X )$ is constant along $c$. If $k=0$, then $c$ is a geodesic. If $k\neq 0$,   set
\be\label{Eq;metric comp 4}
Y=\frac{1}{k} \nabla_{c'}X,
\ee
then $Y$ is a unit vector field which satisfies Eq. (\ref{Eq;cov der of X 1}). The covariant derivative of Eq. (\ref{Eq;metric comp 4}) and using Eq. (\ref{Eq;circle equ+1})  yields to
\begin{eqnarray*}
\nabla_{c'}Y=\frac{1}{k} \nabla_{c'}\nabla_{c'}X = -k X.
\end{eqnarray*}
Thus we have Eqs. (\ref{Eq;cov der of X 1}) and (\ref{Eq;cov der of X 2}), hence $c$ is a circle. This completes the proof.
\end{proof}

\bigskip
For a curve $c=c(s)$ paramertized by arc-length $s$,  $c'(s):=\frac{dc}{ds}(s)$ is the unit tangent vector along $c$. Let
$$  c''(s):=\nabla_{c'} c', \ \ \ \ c''(s):=\nabla_{c'}\nabla_{c'} c', \ \ \ \
c'''(s):=\nabla_{c'}\nabla_{c'}\nabla_{c'} c'.$$
We can express (\ref{Eq;circle equ+1}) as follows
\be
c''' + g(c'', c'') c' =0. \label{Eq;circle equ+3}
\ee
Equivalently, differential equation of a circle is given by
  \begin{equation}\label{Eq: 1.9+1}
  c''' = -k^2 c',
\end{equation}
 where $k=\sqrt{ g(c'',c'')}$ is the constant  first Frenet curvature. Hence,
$c(s)$ is a circle if and only if
$c'''$ is a tangent vector field along $c$, or equivalently $c'''$ is a scalar multiple of $c'$ or $\dot c$.

 If $c=c(t)$ is parameterized  by an arbitrary parameter $t$, we denote its successive covariant derivatives by $\dot c:=\frac{dc}{dt}$, $\ddot c=:\nabla_{\dot{c}}\dot{c}$ and $\dddot c:= \nabla_{\dot{c}}\nabla_{\dot{c}} \dot{c}$.
We have the following successive relations between successive covariant derivatives.
\begin{eqnarray}\label{Eq;successive1 deri}
\dot c &=&  |\dot c| \ c', \\
\label{Eq;successive2 deri}
 \ddot c&=&|\dot c|^2 \  c'' + \frac{g(\dot c,\ddot c)}{|\dot c|}\ c',\\
 \label{Eq;successive3 deri}
  \dddot c&=&|\dot c|^3 \  c''' +3{ g(\dot c,\ddot c)}c''+ \frac{d}{dt}\bigg(\frac{g( \dot c,\ddot c)}{|\dot c|}\bigg)\ c'.
\end{eqnarray}
For an arbitrary parameter $t$ we have the following lemma.
\begin{lem}\label{lem;circle}
Let $(M,g)$ be a Finsler manifold and $c(t)$ a curve on $M$. Then $c(t)$ is a  circle with respect to the $g$, if and only if the vector field
$$V :=\dddot c-3\frac{ g(\dot c,\ddot c)}{g(\dot c,\dot c)}\ddot c,$$
is a tangent vector field along $c$ or equivalently a multiple of $\dot c$ or $c'$.
\end{lem}
\begin{proof}
It follows from (\ref{Eq;successive2 deri}) and (\ref{Eq;successive3 deri}) that
\[ \dddot c-3\frac{ g(\dot c,\ddot c)}{g(\dot c,\dot c)}\ddot c = |\dot{c}|^3 c''' + \Big \{ \frac{d}{dt}\bigg(\frac{g( \dot c,\ddot c)}{|\dot c|}\bigg) -3 \frac{ g(\dot{c}, \ddot{c})^2}{ g(\dot{c}, \dot{c})^{3/2} }
\Big \} c' .\]
Thus $c'''$ is parallel to $c'$ if and only if $\dddot c-3\frac{ g(\dot c,\ddot c)}{g(\dot c,\dot c)}\ddot c $ is proportional to $c'$
\end{proof}
Contrary to the Euclidean circle, the general notion of circle in Riemannian geometry as well as in  Finsler geometry, called here, simply  circle,  is not required that a circle be a closed curve. Although it may happen in some cases as small circles or helicoid curves on the sphere. In general, similar to the Riemannian circles, they are  spiral curves on the  subordinate spaces, for instance, spiral curves on cylindrical surfaces, conical surfaces and so on. Moreover, their length are not required to be bounded as in closed circle in Euclidean spaces.
\section{Circle-preserving diffeomorphisms}
   A  local diffeomorphism of Finsler manifolds is said to be  \emph{circle-preserving}, if it maps circles into circles.
More precisely let $M$  be a differentiable  manifold,  $g$ a Finsler metric on $M$,  $c(s)$ a $C^\infty$  arc length parameterized curve in a neighborhood $ U\subset M$ and $\delta/ ds$ the Cartan covariant derivative along $c$, compatible with $g$.

 Let $\phi: M \longrightarrow M$ be a local diffeomorphism on $M$, then it induces a second  Finsler metric $\bar g$  and a Cartan covariant derivative $\delta /d\bar s$  along $\bar c$ on $(M,\bar g)$ on some neighborhood $\bar U$ of $M$. Here, we denote the induced Finsler manifold by $(M,\bar g)$, in the sequel.
  We say that  the local diffeomorphism $\phi: (M,g)  \longrightarrow (M,\bar g)$ \emph{preserves circles}, if it maps circles to  circles.

 Let $ c(s)$ be a circle and  $\bar c(\bar s)$ its image by $\phi$,  where $\bar s=\bar s(s)$.  Then using definite positiveness of $g$ and $\bar g$ and the related fundamental forms
 \be \label{Eq;arclenth 1}
 ds^2= g_{ij}(x,x')dx^idx^j \quad \textrm{and}\quad d\bar s^2= \bar g_{ij}(x,x')dx^idx^j,
 \ee
  respectively,  we can establish a relation between  $s$ and $\bar s$ and their derivatives.
  We have  $\frac{\delta}{d\bar s}=\frac{\delta}{ds} . \frac{ds}{d\bar s}$, where
\be\label{Eq;arclenth}
\frac{d\bar s}{ds}=\sqrt{\bar g_{jk}(x,x')\frac{d x^j}{ds}\frac{d x^k}{ds}}\neq 0, \quad
\frac{d s}{d \bar s}=\sqrt{ g_{jk}(x,x')\frac{d x^j}{d\bar s}\frac{d x^k}{d\bar s}}\neq 0.
\ee
If we have $d\bar s =e^{\sigma} ds$, or equivalently  by means of Eq. (\ref{Eq;arclenth 1}), if   $\bar g=e^{2\sigma} g$ or  $\bar F=e^{\sigma} F$, where $\sigma$ is a scalar function on $M$, then two Finsler structures $\bar F$ and $F$ are said to be {\it conformal}.

 \section{Circles in a Minkowski space}
Let $(V, F)$ be a Minkowski space where $V$ is a vector space and $F$ is a Minkowski norm on $V$.
In a standard coordinate system in $V$,
\be\label{Eq;Minkowski}
 G^i =0, \ \ \ \ \ N^i_j =0.
 \ee
Then for a vector field $X= X^i(t)\frac{\pa }{\pa x^i} |_{c(t)}$ along a curve $c(t)$,    the Cartan covariant derivative $\nabla_{\dot c} X= \frac{\delta X^i}{dt} \frac{\pa }{\pa x^i} |_{c(t)}$  given by Eq. (\ref{CartanCovDerAlongC1}) reduces to
\begin{eqnarray}\label{CartanCovDerAlongC2}
\frac{\delta X^i}{ dt} & = &\frac{d X^i}{ dt}+ \Gamma_{kh}^i X^k\frac{dx^h}{dt}.
\end{eqnarray}
In particular, for $X= c'$, we have
\[ \frac{\delta X^i}{ds} = \frac{d^2 x^i}{ds^2}+ \Gamma_{kh}^i \frac{dx^k}{ds}\frac{dx^h}{ds}= \frac{d^2 x^i}{ds^2}+ G^i,\]
 where we have used  $\Gamma^i_{jk} \frac{dx^j}{dt} \frac{dx^k}{dt}= \gamma^i_{jk} \frac{dx^j}{dt} \frac{dx^k}{dt}=G^i,$ for which $\gamma^i_{jk}$ are the formal Christoffel symbols. Thus Eq. (\ref{Eq;Minkowski}) yeilds
   \[ \frac{\delta X^i}{ds} = \frac{d^2 x^i}{ds^2}.\]

Therefore in a Minkowski space, a curve $c(s)$ with arc-length parameter $s$ is a circle if and only if
\be
\frac{d^3 x^i}{ds^3} + k^2
\frac{dx^i}{ds}=0, \label{mmm}
\ee
where $k$ is a constant.
In this case $ g_{hk}\frac{d^2 x^h}{ds^2} \frac{d x^k}{ds} =0$ and
$k =\sqrt{ g_{hk}\frac{d^2 x^h}{ds^2} \frac{d^2 x^k}{ds^2} }$.

Now let us take a look at circles in a special Minkowski space $(\R^2, F_b)$, where
\[  F_b :=\sqrt{u^2+v^2} + bu ,\]
where $b$ is a positive constant with $ 0 < b < 1$. $(\R^2, F_b)$ is called a {\it Randers plane}.
Consider a curve $c(s) = (x(s), y(s))$ in $\R^2$ with unit speed, namely, $c'(s)=(x'(s), y'(s))$ is a unit vector.
Thus
\[   \sqrt{ x'(s)^2 + y'(s)^2} + b x'(s)= 1.\]
We can let
\[  x'(s) = \frac{ \cos \theta (s) - b}{1-b^2}, \ \ \ \ \ \ y'(s)= \frac{\sin \theta(s)}{\sqrt{1-b^2}}.\]
where $\theta(s)$ is a smooth function. Since $b$ is bounded, components of $c'(s)$
are well defined and one can find out explicitly equation of $c(s)$, the unit circle in the Randers plane $(\R^2, F_b)$.
\section{Vogel Theorem in Finsler geometry}

\bigskip

Let $\phi: (M, g) \longrightarrow (\bar{M}, \bar{g})$ be a  diffeomorphism.
  We say that $\phi$  \emph{preserves circles}, if it maps circles to  circles.  More precisely, if $c(s)$ is a circle in $(M, g)$, where $s$ is the arc-length of $c$ with respect to $g$, then $\bar{c}(\bar{s}):=\phi \circ c (s(\bar{s}))$ is a circle in $(\bar{M}, \bar{g})$, where  $s = s (\bar{s})$ is the arc-length of $\phi \circ c$ with respect to $\bar{g}$.

We recall the following lemma from linear algebra which will be used  in the sequel.
\begin{lem}\label{lem;linear Algebra}
Let $F$ and $G$ be the two bilinear symmetric  forms on $\R^n$, satisfying
\begin{itemize}
\item $F$ and $G$ are definite positive.
\item $F$ and $G$ are defined on $\R^n\times\R^n\longrightarrow\R,$ such that
$F(X,Y)=0$, $\forall X,Y \in\R^n,$ with
\begin{eqnarray}\label{Eq;bilinearForms}
  &&G(X,X)\neq0,  G(Y,Y)\neq0 , \textrm{and } G(X,Y)=0,
 \end{eqnarray}
 \end{itemize}
  then there is a positive real number $\alpha$ such that
$F=\alpha G.$
\end{lem}
\begin{proof}
Let $\{ e_i\},$  be an orthonormal basis on $\R^n$ such that $G (e_i, e_j)=\delta_{ij}$ where $i,j=1,...,n$.  Eq.(\ref{Eq;bilinearForms}) with definite positiveness of $F$ and $G$ imply  that there is a positive real number $\alpha_i$ such that $F(e_i,e_j)=\alpha_i\delta_{ij},$ and hence
 \be \label{Eq;bilinearForms2}
F(e_i,e_j)=\alpha_iG (e_i, e_j).
 \ee
 Let $a,b\in \R-\{0\}$ with $a^2\neq b^2$ , then for $i\neq j$ we have
\begin{eqnarray*}
&&G(ae_i + be_j,ae_i + be_j) = a^2 + b^2 \neq 0,\\
&&G(be_i - ae_j,be_i - ae_j) = b^2 + a^2 \neq 0,\\
&&G(ae_i + be_j,be_i - ae_j) = 0.
\end{eqnarray*}
This equation together with Eqs.(\ref{Eq;bilinearForms}) and (\ref{Eq;bilinearForms2}) for $i\neq j$ imply $$0=F(ae_i + be_j,be_i - ae_j)=ab(\alpha_iG(e_i,e_i)-\alpha_jG(e_j,e_j))$$ and hence $0 = ab(\alpha_i-\alpha_j)$. Therefore we obtain $\alpha_i=\alpha_j=\alpha$, and  $F=\alpha G,$ which completes the proof.
\end{proof}
 Next we prove the following theorem.
 \setcounter{thm}{0}
 \begin{thm}\label{thm;Vogel1}
 Every circle-preserving local diffeomorphism of a Finsler manifold is conformal.
\end{thm}
\begin{proof} Without loss of generality we can consider two Finsler metrics $g$ and $\bar{g}$ on the same manifold.
Fix a point $p\in M$. For arbitrary two unit vectors $X, Y\in T_pM$ such that $Y$ is orthogonal to $X$ with respect to $g=g_{_X}$, let ${\cal C}=\{c_k|k\in \R\}$ be a family of  circles with the constant  curvature $k$ passing through a fixed point $c_{_k}(0) = p$ on $(M,g)$ such that
\be \label{Eq;circle equ}
  \frac{dc}{ds}(0) = X,  \quad \textrm{and}\quad \nabla_{c'}X (0) =kY.
  \ee
We are going to show that $\bar{g}(X, Y)=0$, where $\bar{g}:=\bar{g}_{_X}$.

Since $c$ is assumed to be a circle with respect to the Finsler metric $g$, Eq. (\ref{Eq: 1.9+1}) yields, $c'''$ is a multiple of $c'$.
By Lemma \ref{lem;circle}, $c$ is a circle with respect to $\bar g$ if and only if $\dddot c $ and $\ddot c $ are parallel to $\dot c$ or $c'$.

On the other hand, by virtue of Eq. (\ref{Eq;successive2 deri}), $\ddot c $ is parallel to $\dot c$, if and only if $c''$ so does.

By means of Eq. (\ref{Eq;successive3 deri}), we can see that $\dddot c$ is  parallel to $\dot c$ if and only if  $\ddot c $ so does. Denote the  second term in right-hand side of Eq. (\ref{Eq;successive3 deri}), by ${\cal \overline{ W}}:=3\bar g(\dot c,\ddot c) c''$. Therefore, $c$ is a circle with respect to $\bar g$ if and only if ${\cal \overline{ W}}$ is parallel to  $c'= X$ at the point $p=c(0)$. At  $p$, by involving Eqs. (\ref{Eq;successive1 deri}) and (\ref{Eq;circle equ})  we have $\dot c=c'\bar g(\dot c,\dot c)^{1/2}=X\bar g(\dot c,\dot c)^{1/2}$, where  $\bar g(\dot c,\dot c)\neq0$ is constant by means of Eq. (\ref{Eq;arclenth}). Hence, Eq.
(\ref{Eq;circle equ}) yields $\ddot c=\bar{\nabla}_{\dot{c}} \dot{c}=\frac{d }{ds}(\bar g(\dot c,\dot c)^{1/2} X)\frac{ds}{dt}$ and $\ddot c=k Y\bar g(\dot c,\dot c)$. Therefore we obtain
\begin{eqnarray*}
{\cal \overline{ W}} = 3\bar g(\dot c,\dot c)^{3/2}\bar g(X , Y )Y \ k^2.
\end{eqnarray*}
  Hence, $c$ is a circle with respect to $\bar g$ if and only if the vector field ${\cal \overline{ W}}$ is parallel to $X$ or equivalently, $\bar g(X , Y )Y$ is parallel to $X$ for every $X\in T_pM$ and every $Y\in T_pM$ orthonormal to $X$. This implies $\bar g(X , Y )=0$ whenever $g(X,Y)=0$ and by Lemma \ref{lem;linear Algebra}, there is a positive scalar $\alpha^2$ where $\bar g=\alpha^2 g$. Hence, the Finsler metrics $\bar g(x,x')$ and $ g(x,x')$ are conformally related.
\end{proof}

  A geodesic circle is a curve for which the first Frenet curvature $k_1$ is constant and the second curvature $k_2$ vanishes. In the other words a geodesic circle is a torsion free constant curvature curve. In Riemannian geometry as well as in Finsler geometry, a concircular transformation is defined to be a conformal transformation which preserves geodesic circles.

  By replacing the positive scalar $\alpha$ in proof of the  theorem \ref{thm;Vogel1} by $\alpha=e^\sigma$, we get $\bar g=e^{2\sigma} g$, or equivalently $d\bar s =e^{\sigma} ds$ where $\sigma$ is a scalar function on $M$. Therefore, as a corollary of Theorem \ref{thm;Vogel1} we have
  \begin{thm}\label{thm;Vogel2 }
  Every local diffeomorphism of a Finsler manifold which preserve geodesic circles  is conformal.
\end{thm}
This result shows that in the definition of concircular transformations the conformal assumption is not necessary.

\textbf{Acknowledgments}. The authors take this opportunity  to express their sincere gratitude to the Professor H. Akbar-Zadeh for his suggestions on this work and his contributions on Finsler geometry.


\end{document}